\documentclass[12pt,twoside]{amsart}
\usepackage[all]{xy}

\title{Notes on toric varieties from Mori theoretic viewpoint, II} 
\author{Osamu Fujino and Hiroshi Sato} 
\subjclass[2010]{Primary 14M25; Secondary 14E30.}
\date{2018/3/26, version 0.28}
\keywords{toric Mori theory, lengths of extremal 
rays, Fujita's conjectures, vanishing theorem}
\address{Department of Mathematics, Graduate School of 
Science, Osaka University, Toyonaka, Osaka 560-0043, Japan}
\email{fujino@math.sci.osaka-u.ac.jp}
\address{Department of Applied Mathematics, Faculty of Sciences, 
Fukuoka University, 8-19-1, Nanakuma, Jonan-ku, Fukuoka 814-0180, Japan}
\email{hirosato@fukuoka-u.ac.jp}

\makeatletter
    
    \@addtoreset{equation}{section}
\makeatother

\newcommand{\Pic}[0]{{\operatorname{Pic}}}
\newcommand{\Int}[0]{{\operatorname{Int}}}
\newcommand{\Star}[0]{{\operatorname{Star}}}

\newcommand{\mult}[0]{{\operatorname{mult}}}
\newcommand{\codim}[0]{{\operatorname{codim}}}
\newcommand{\xSupp}[0]{{\operatorname{Supp}}}
\newcommand{\Nef}[0]{{\operatorname{Nef}}}
\newcommand{\PE}[0]{{\operatorname{PE}}}
\newcommand{\Amp}[0]{{\operatorname{Amp}}}
\newcommand{\NE}[0]{{\operatorname{NE}}}

\newtheorem{thm-a}{Theorem}[section]
\newtheorem{thm}{Theorem}[subsection]
\newtheorem{lem}[thm]{Lemma}
\newtheorem{cor}[thm]{Corollary}
\newtheorem{cor-a}[thm-a]{Corollary}
\newtheorem{prop}[thm]{Proposition}
\newtheorem*{claim}{Claim}

\theoremstyle{definition}
\newtheorem{ex}[thm]{Example}
\newtheorem{defn}[thm]{Definition}
\newtheorem{rem}[thm]{Remark}
\newtheorem*{ack}{Acknowledgments}       
\newtheorem{say}[thm]{}
\newtheorem{case}{Case}
\newtheorem{step}{Step}

\begin{document}
\bibliographystyle{amsalpha+}

\begin{abstract}
We give new estimates of lengths of 
extremal rays of birational type for toric 
varieties. 
We can see that our new estimates are the best by constructing 
some examples explicitly. 
As applications, 
we discuss the nefness and pseudo-effectivity of 
adjoint bundles of projective 
toric varieties. 
We also treat some generalizations of Fujita's freeness and very ampleness 
for toric varieties. 
\end{abstract}

\maketitle

\tableofcontents
\section{Introduction}

The following theorem is one of the main results of this paper. 
Our proof of Theorem \ref{f-thm1.1} 
uses the framework of the toric Mori theory developed by 
\cite{reid}, \cite{fujino-notes}, 
\cite{fujino-equiv}, \cite{fujino-osaka}, \cite{fujino-sato}, and so on. 

\begin{thm-a}[Theorem \ref{f-thm4.2.3} and 
Corollary \ref{f-cor4.2.4}]\label{f-thm1.1} 
Let $X$ be a $\mathbb Q$-Gorenstein 
projective toric $n$-fold and let $D$ be an ample Cartier divisor on $X$. 
Then $K_X+(n-1)D$ is pseudo-effective if and only if $K_X+(n-1)D$ is 
nef. 
In particular, if $X$ is Gorenstein, 
then \[H^0(X, \mathcal O_X(K_X+(n-1)D))\ne 0\] 
if and only if the complete linear system $| K_X+(n-1)D|$ is 
basepoint-free. 
\end{thm-a}

This theorem was inspired by Lin's paper (see \cite{lin}). 
Our proof of Theorem \ref{f-thm1.1} depends on the following new 
estimates of lengths 
of extremal rays of birational type for toric varieties. 

\begin{thm-a}[Theorem \ref{f-thm3.2.1}]\label{f-thm1.2} 
Let $f:X\to Y$ be a projective toric morphism with $\dim X=n$. 
Assume that $K_X$ is $\mathbb Q$-Cartier. 
Let $R$ be a $K_X$-negative 
extremal ray of $\NE(X/Y)$ and let $\varphi_R:X\to W$ 
be the contraction morphism associated to $R$. We put 
\[l(R)=\min_{[C]\in R} (-K_X\cdot C).\] 
and call it the length of $R$. 
Assume that $\varphi_R$ is birational. 
Then we obtain 
\[
l(R)<d+1, 
\]
where \[d=\max_{w\in W} \dim \varphi^{-1}_R(w)\leq n-1.\] 
When $d=n-1$, we have a sharper inequality 
\[
l(R)\leq d=n-1. 
\] 
In particular, if 
$l(R)=n-1$, 
then $\varphi_R:X\to W$ can be described as follows. 
There exists a torus invariant smooth point $P\in W$ such that 
$\varphi_R:X\to W$ 
is a weighted blow-up at $P$ with the weight $(1, a, \cdots, a)$ for some 
positive integer $a$. 
In this case, the exceptional locus $E$ of $\varphi_R$ is a torus invariant 
prime divisor and is isomorphic to $\mathbb P^{n-1}$. 
Moreover, $X$ is $\mathbb Q$-factorial in a neighborhood of 
$E$. 
\end{thm-a}

Theorem \ref{f-thm1.2} supplements \cite[Theorem 0.1]{fujino-notes} 
(see also \cite[Theorem 3.13]{fujino-equiv}). 
We will see that the estimates obtained in Theorem \ref{f-thm1.2} 
are the best by constructing some examples explicitly 
(see Examples \ref{f-ex3.3.1} and \ref{f-ex3.3.2}). 
For lengths of extremal rays for non-toric varieties, 
see \cite{kawamata}. As an application 
of Theorem \ref{f-thm1.2}, we can prove the following theorem 
on lengths of extremal rays for $\mathbb Q$-factorial toric varieties. 

\begin{thm-a}[Theorem \ref{f-thm3.2.9}]\label{f-thm1.3}
Let $X$ be a $\mathbb Q$-Gorenstein 
projective toric $n$-fold with $\rho (X)\geq 2$. 
Let $R$ be a $K_X$-negative extremal ray of 
$\NE(X)$ such that 
\[
l(R)=\min _{[C]\in R}(-K_X\cdot C)>n-1. 
\] 
Then the extremal contraction $\varphi_R:X\to W$ associated to 
$R$ is a $\mathbb P^{n-1}$-bundle over $\mathbb P^1$. 
\end{thm-a}

As a direct easy consequence of Theorem \ref{f-thm1.3}, we obtain 
the following corollary, which supplements Theorem \ref{f-thm1.1}. 

\begin{cor-a}[Corollary \ref{f-cor4.2.5}]\label{f-cor1.4}  
Let $X$ be a $\mathbb Q$-Gorenstein projective 
toric $n$-fold and let $D$ be an ample Cartier divisor 
on $X$. 
If $\rho (X)\geq 3$, then $K_X+(n-1)D$ is always nef. 
More precisely, 
if $\rho (X)\geq 2$ and $X$ is not a $\mathbb P^{n-1}$-bundle 
over $\mathbb P^1$, then $K_X+(n-1)D$ is nef. 
\end{cor-a} 

In this paper, we also give some generalizations of 
Fujita's freeness and very ampleness for toric varieties  
based on our powerful vanishing 
theorem (see \cite{fujino-vanishing} and \cite{fujino-toric}). 
As a very special case of our generalization of Fujita's freeness for toric varieties 
(see Theorem \ref{f-thm4.1.1}), 
we can easily recover some parts of Lin's theorem (see \cite[Main Theorem A]{lin}). 

\begin{thm-a}[Corollary \ref{f-cor4.1.2}]\label{f-thm1.5} 
Let $X$ be an $n$-dimensional projective toric variety 
and let $D$ be an ample Cartier divisor on $X$. Then
the reflexive sheaf $\mathcal O_X(K_X+(n+1)D)$ is generated
by its global sections.
\end{thm-a}

By the same way, we can obtain a generalization of Fujita's very ampleness 
for toric varieties (see Theorem \ref{f-thm4.1.8}). 

\begin{thm-a}[Theorem \ref{f-thm4.1.6}]\label{f-thm1.6} 
Let $f:X\to Y$ be a proper surjective toric 
morphism, let $\Delta$ be a reduced torus invariant 
divisor on $X$ such that $K_X+\Delta$ is Cartier, 
and let $D$ be an $f$-ample Cartier divisor on $X$. 
Then $\mathcal O_X(K_X+\Delta+kD)$ is $f$-very ample for 
every $k\geq \max_{y\in Y} \dim f^{-1}(y) +2$.  
\end{thm-a}

For the precise statements of our generalizations of 
Fujita's freeness and very ampleness for toric varieties, 
see Theorems \ref{f-thm4.1.1} and \ref{f-thm4.1.8}. 
We omit them here since they are technically complicated. 

This paper is organized as follows. 
In Section \ref{f-sec2}, we collect some basic definitions and 
results. In subsection \ref{f-subsec2.1}, we 
explain the basic concepts of the toric geometry. 
In subsection \ref{f-subsec2.2}, we recall the definitions of 
{\em{the Kleiman--Mori cone}}, {\em{the nef cone}}, {\em{the ample cone}}, 
and {\em{the pseudo-effective cone}} for toric projective morphisms, 
and some related results. Section \ref{f-sec3} is the main part of this paper. 
After recalling the known estimates of lengths of extremal rays 
for projective toric varieties in 
subsection \ref{f-subsec3.1}, we give 
new estimates of lengths of extremal rays of toric birational contraction 
morphisms in subsection \ref{f-subsec3.2}. 
In subsection \ref{f-subsec3.3}, we see that the estimates 
obtained in subsection \ref{f-subsec3.2} are the best by 
constructing some examples explicitly. 
Section \ref{f-sec4} treats 
Fujita's freeness and very ampleness for toric varieties. 
The results in subsection \ref{f-subsec4.1} depend 
on our powerful vanishing theorem 
for toric varieties and are independent of our estimates of lengths of 
extremal rays for toric varieties. Therefore, 
subsection \ref{f-subsec4.1} is independent of 
the other parts of this paper. 
In subsection \ref{f-subsec4.2}, we discuss Lin's problem (see \cite{lin}) related to 
Fujita's freeness for toric varieties. We use our new estimates of 
lengths of extremal rays in this subsection. 
Subsection \ref{f-subsec4.3} is a supplement to Fujita's paper:~\cite{fujita}. 
This paper contains various supplementary results for \cite{fujita}, 
\cite{fulton}, \cite{lin}, and so on. 

\begin{ack}\label{f-ack}
The first author was partially supported by JSPS KAKENHI Grant 
Numbers JP16H03925, JP16H06337. 
If the first author remembers correctly, he prepared 
a preliminary version of this paper around 2006 in Nagoya. 
Then his interests moved to the minimal model program. 
In 2011, he revised it in Kyoto. 
The current version was written in Osaka. 
He thanks the colleagues in Nagoya, Kyoto, and 
Osaka very much.  
\end{ack}

We will work over an arbitrary algebraically closed field throughout this 
paper. For the standard notations of the minimal model program, 
see \cite{fujino-fundamental} and \cite{fujino-foundation}. 
For the toric Mori theory, we recommend the reader to see \cite{reid}, 
\cite[Chapter 14]{matsuki}, \cite{fujino-notes}, and \cite{fujino-sato} (see 
also \cite{cls}). 

%%%%%%%%%%%%%%%%%%%%%%%%%%%%%%%%%%%%
\section{Preliminaries}\label{f-sec2}

This section collects some basic definitions and results. 

\subsection{Basics of the toric geometry}\label{f-subsec2.1}
In this subsection, we recall the basic notion of toric varieties 
and fix the notation. For the details, see \cite{oda}, \cite{fulton}, 
\cite{reid}, or \cite[Chapter 14]{matsuki} (see also \cite{cls}). 

\begin{say}\label{f-say2.1.1}
Let $N\simeq \mathbb Z^n$ be a lattice of rank $n$. 
A toric variety $X(\Sigma)$ is associated to a {\em{fan}} $\Sigma$, 
a correction of convex cones $\sigma\subset N_\mathbb R =
N\otimes _{\mathbb Z}\mathbb R$ satisfying: 
\begin{enumerate}
\item[(i)] 
Each convex cone $\sigma$ is a rational polyhedral cone in the sense that  
there are finitely many 
$v_1, \cdots, v_s\in N\subset N_{\mathbb R}$ such 
that 
\[
\sigma=\{r_1v_1+\cdots +r_sv_s; \ r_i\geq 0\}=
:\langle v_1, \cdots, v_s\rangle, 
\] 
and 
it is strongly convex in the sense that  
\[
\sigma \cap -\sigma=\{0\}. 
\]
\item[(ii)] Each face $\tau$ of a convex cone $\sigma\in \Sigma$ 
is again an element in $\Sigma$. 
\item[(iii)] The intersection of two cones in $\Sigma$ is a face of 
each. 
\end{enumerate}

\begin{defn}\label{f-def2.1.2}
The {\em{dimension}} $\dim \sigma$ of a cone $\sigma$ is 
the dimension of the linear space 
$\mathbb R\cdot \sigma=\sigma +(-\sigma)$ spanned 
by $\sigma$. 

We define the sublattice $N_{\sigma}$ of 
$N$ generated (as a subgroup) by $\sigma\cap N$ as 
follows: 
\[
N_{\sigma}:=\sigma\cap N+(-\sigma\cap N). 
\]

If $\sigma$ is a $k$-dimensional simplicial 
cone, and $v_1,\cdots, v_k$ are the 
first lattice points along the edges of $\sigma$, 
the {\em{multiplicity}} of $\sigma$ is defined 
to be the {\em{index}} of the lattice 
generated by the $\{v_1, \cdots, v_k\}$ in the lattice $N_{\sigma}$; 
\[
\mult (\sigma):=[N_{\sigma}:\mathbb Zv_1+\cdots +
\mathbb Zv_k]. 
\] 
We note that the affine toric variety 
$X(\sigma)$ associated to the cone $\sigma$ is smooth if and only 
if $\mult (\sigma)=1$. 
\end{defn}
\end{say}

The following is a well-known fact. See, for example, 
\cite[Lemma 14-1-1]{matsuki}. 

\begin{lem}\label{f-lem2.1.3}
A toric variety $X(\Sigma)$ is $\mathbb Q$-factorial 
if and only if each cone $\sigma\in \Sigma$ is simplicial. 
\end{lem}

\begin{say}\label{f-say2.1.4}
The {\em{star}} of a cone $\tau$ can be defined abstractly 
as the set of cones $\sigma$ in $\Sigma$ that 
contain $\tau$ as a face. Such cones $\sigma$ are 
determined by their images in 
$N(\tau):=N/{N_{\tau}}$, that is, by 
\[
\overline \sigma=\sigma+(N_{\tau})_{\mathbb R}/ 
(N_{\tau})_{\mathbb R}\subset N(\tau)_{\mathbb R}. 
\] 
These cones $\{\overline \sigma ; \tau\prec \sigma\}$ 
form a fan in $N(\tau)$, and we denote this fan by 
$\Star(\tau)$. 
We set $V(\tau)=X(\Star (\tau))$, that is, the toric variety associated to 
the fan $\Star (\tau)$. 
It is well known that $V(\tau)$ is an $(n-k)$-dimensional 
closed toric subvariety of $X(\Sigma)$, where $\dim \tau=k$. 
If $\dim V(\tau)=1$ (resp.~$n-1$), then we call $V(\tau)$ 
a {\em{torus invariant curve}} (resp.~{\em{torus invariant 
divisor}}). 
For the details about the correspondence between $\tau$ and 
$V(\tau)$, see \cite[3.1 Orbits]{fulton}. 
\end{say}

\begin{say}[Intersection theory for $\mathbb Q$-factorial 
toric varieties]\label{f-say2.1.5}
Assume that $\Sigma$ is simplicial. 
If $\sigma, \tau\in \Sigma$ span $\gamma\in \Sigma$ with 
$\dim \gamma=\dim \sigma +\dim \tau$, 
then 
\[
V(\sigma)\cdot V(\tau)=\frac{\mult (\sigma)\cdot \mult(\tau)}
{\mult (\gamma)} V(\gamma)
\] 
in the {\em{Chow group}} $A^{*}(X)_{\mathbb Q}$. 
For the details, see \cite[5.1 Chow groups]{fulton}. 
If $\sigma$ and $\tau$ are contained in no cone of 
$\Sigma$, then $V(\sigma)\cdot V(\tau)=0$. 
\end{say}

\subsection{Cones of divisors}\label{f-subsec2.2} 
In this subsection, we explain various cones of divisors 
and some related topics. 

\begin{say}\label{f-say2.2.1}
Let $f:X\to Y$ be a proper toric 
morphism; 
a $1$-cycle of $X/Y$ is a formal sum $\sum a_iC_i$ with 
complete curves $C_i$ in the fibers of $f$, and 
$a_i\in \mathbb Z$. 
We put 
\[
Z_1(X/Y):=\{1\text{-cycles of} \ X/Y\}, 
\]
and 
\[
Z_1(X/Y)_{\mathbb R}:=
Z_1(X/Y)\otimes \mathbb R.  
\]
There is a pairing 
\[
\Pic (X)\times Z_1(X/Y)_{\mathbb R}
\to \mathbb R
\]
defined by 
$(\mathcal L, C)\mapsto \deg _C\mathcal L$, extended 
by bilinearity. 
We define 
\[
N^1(X/Y):=(\Pic (X)\otimes \mathbb R)/\equiv
\]
and 
\[
N_1(X/Y):= Z_1(X/Y)_{\mathbb R}/\equiv, 
\] 
where the {\em numerical equivalence} $\equiv$ 
is by definition the smallest equivalence relation which 
makes $N^1$ and $N_1$ into dual 
spaces. 

Inside $N_1(X/Y)$ there is a distinguished cone of effective 
$1$-cycles of $X/Y$, 
\[
{\NE}(X/Y)=\{\, Z\, | \ Z\equiv 
\sum a_iC_i \ \text{with}\ a_i\in \mathbb R_{\geq 0}\}
\subset N_1(X/Y), 
\] 
which is usually called the {\em{Kleiman--Mori cone}} of $f:X\to Y$. 
It is known that $\NE(X/Y)$ is a rational polyhedral cone. 
A face $F\prec {\NE}(X/Y)$ is called an {\em{extremal face}} 
in this case. 
A one-dimensional extremal face is called an {\em{extremal 
ray}}. 

We define the {\em{relative Picard number}} $\rho(X/Y)$ by
\[
\rho (X/Y):=\dim _{\mathbb Q}N^1(X/Y)< \infty. 
\]
An element $D\in N^1(X/Y)$ is called {\em{$f$-nef}} if 
$D\geq 0$ on ${\NE}(X/Y)$. 

If $X$ is complete and $Y$ is a point, then we 
write ${\NE}(X)$ and $\rho(X)$ 
for ${\NE}(X/Y)$ and $\rho(X/Y)$, respectively. 
We note that $N_1(X/Y)\subset N_1(X)$, and $N^1(X/Y)$ 
is the corresponding 
quotient of $N^1(X)$. 

From now on, we assume that $X$ is complete.
We define the {\em{nef cone}} $\Nef (X)$, the
{\em{ample cone}} $\Amp(X)$, and 
the
{\em{pseudo-effective cone}} $\PE(X)$ in $N^1(X)$ as follows.
\[
\Nef(X)=\{D\, |\,  D \ \text{is nef}\},
\]
\[
\Amp(X)=\{D\, | \, D\ \text{is ample}\}
\]
and
\[
\PE(X)=\left\{D\, \left|  \begin{array}{l} 
D\equiv \sum a_i D_i \ \text{such that
$D_i$ is an effective}\\ \text{Cartier divisor 
and $a_i\in \mathbb R_{\geq 0}$ for every $i$
}\end{array}\right. \right\}.
\]
It is not difficult to see that $\PE(X)$ is a rational polyhedral cone in
$N^1(X)$ since $X$ is toric.
It is easy to see that \[\Amp(X)\subset \Nef (X)\subset \PE(X).\] 
The reader can find various examples of 
cones of divisors and curves in \cite{kleiman}, \cite{fujino-payne}, 
and \cite{fujino-sato2}. Although we do not use it explicitly in this paper, 
Lemma \ref{f-lem2.2.2} is a very important property of toric varieties. 

\begin{lem}\label{f-lem2.2.2}
Let $X$ be a projective toric variety and let $D$ be a 
$\mathbb Q$-Cartier $\mathbb Q$-divisor on $X$. 
Then $D$ is pseudo-effective if and only if $\kappa (X, D)\geq 0$, 
that is, there exists a positive integer $m$ such that $mD$ is Cartier and that 
\[H^0(X, \mathcal O_X(mD))\ne 0. \]
\end{lem}
\begin{proof}
It is sufficient to prove that $\kappa (X, D)\geq 0$ when 
$D$ is pseudo-effective. 
Let $f:Y\to X$ be a projective toric birational morphism 
from a smooth projective toric variety $Y$. 
By replacing $X$ and $D$ with $Y$ and $f^*D$, we may assume that 
$X$  is a smooth projective toric variety. 
In this case, it is easy to see that $\PE(X)$ is spanned 
by the numerical equivalence classes 
of torus invariant prime divisors. 
Therefore, we can write $D\equiv \sum _i a_i D_i$ where 
$D_i$ is a torus invariant prime divisor 
and $a_i\in \mathbb Q_{>0}$ for every $i$ since $D$ is 
a $\mathbb Q$-divisor. 
Thus, we obtain $D\sim _{\mathbb Q} \sum _i a_i D_i\geq 0$. 
This implies $\kappa (X, D)\geq 0$. 
\end{proof}

The following lemma is well known and is very important. 
We will use it in the subsequent sections repeatedly. 

\begin{lem}\label{f-lem2.2.3} 
Let $f:X\to Y$ be a toric proper morphism 
and let $D$ be an $f$-nef Cartier divisor on $X$. Then $D$ is $f$-free, 
that is, 
\[f^*f_*\mathcal O_X(D)\to \mathcal O_X(D)\] 
is surjective.  
\end{lem}

\begin{proof}
See, for example, \cite[Chapter VI.~1.13.~Lemma]{nakayama}. 
\end{proof}

We close this section with an easy example. 
It is well known that $\NE(X)$ is spanned by the numerical equivalence classes of 
torus invariant irreducible curves. However, the 
dual cone $\Nef (X)$ of $\NE(X)$ is not 
always spanned by the numerical equivalence classes of torus invariant prime 
divisors. 

\begin{ex}\label{f-ex2.2.4}
We consider $\mathbb P^1\times \mathbb P^1$. 
Let $p_i:\mathbb P^1\times \mathbb P^1\to \mathbb P^1$ be 
the $i$-th projection for $i=1, 2$. 
Let $D_1, D_2$ (resp.~$D_3, D_4$) be the 
torus invariant curves in the 
fibers of $p_1$ (resp.~$p_2$). 
Let $X\to \mathbb P^1\times \mathbb P^1$ be the 
blow-up at the point $P=D_1\cap D_3$ and 
let $E$ be the exceptional curve on $X$. 
Let $D'_i$ denote the strict transform of $D_i$ on $X$ for 
all $i$. 
Then $\NE(X)$ is spanned by 
the numerical equivalence classes of $E, D'_1$, 
and $D'_3$. 
On the other hand, $\Nef (X)\subset N^1(X)$ is spanned 
by $D'_2, D'_4$, and $D'_1+D'_3+E$. 
Therefore, the extremal ray of $\Nef (X)$ is not 
necessarily spanned by a torus invariant 
prime divisor. 
\end{ex}
\end{say}

\section{Lengths of extremal rays}\label{f-sec3}  
In this section, we discuss some estimates of lengths of 
extremal rays of toric projective morphisms. 

\subsection{Quick review of the known estimates}\label{f-subsec3.1}

In this subsection, we recall the known estimates of lengths of 
extremal rays for toric varieties. 
The first result is \cite[Theorem 0.1]{fujino-notes} (see 
also \cite[Theorem 3.13]{fujino-equiv}). 

\begin{thm}\label{f-thm3.1.1}
Let $f:X\to Y$ be a projective toric morphism with $\dim X=n$ 
and let $\Delta=\sum \delta_i\Delta_i$ be an $\mathbb R$-divisor 
on $X$ such that $\Delta_i$ is a torus 
invariant prime divisor and $0\leq \delta_i\leq 1$ for every $i$. 
Assume that $K_X+\Delta$ is $\mathbb R$-Cartier. 
Let $R$ be an extremal ray of $\NE(X/Y)$. 
Then there exists a curve $C$ on $X$ such 
that $[C]\in R$ and \[-(K_X+\Delta)\cdot C\leq n+1.\] 
More precisely, we can choose $C$ such that 
\[-(K_X+\Delta)\cdot C\leq n\] unless $X\simeq \mathbb P^n$ 
and $\sum \delta_i<1$. 
We note that if $X$ is complete then we can make $C$ a torus invariant 
curve on $X$. 
\end{thm}

Our proof of Theorems \ref{f-thm3.1.1} and 
\ref{f-thm3.2.1} below heavily depends on 
Reid's description of toric extremal contraction morphisms 
(see \cite{reid} and \cite[Chapter 14]{matsuki}). 

\begin{say}[Reid's description of toric extremal 
contraction morphisms]\label{f-say3.1.2}
Let $f:X\to Y$ be a projective surjective toric morphism from a complete
$\mathbb Q$-factorial toric $n$-fold and let $R$ be 
an extremal ray of $\NE(X/Y)$.
Let $\varphi_R:X\to W$ be the extremal contraction associated to
$R$. We write 
\[
\begin{matrix}
 & A & \longrightarrow &B \\
                   & \cap & &\cap\\
\varphi_R: & X& \longrightarrow &W, \\
\end{matrix}
\] 
where $A$ is the exceptional locus of $\varphi_R$ and
$B$ is the image of $A$ by $\varphi_R$. Then there
exist torus invariant prime divisors $E_1, \cdots, E_{\alpha}$ on $X$ with 
$0\leq \alpha
\leq n-1$ such that $E_i$ is negative on $R$ for $1\leq i \leq \alpha$ 
and that $A$ is $E_1\cap \cdots\cap E_{\alpha}$. In particular, $A$ is
an irreducible torus invariant subvariety of $X$ with
$\dim A=n-\alpha$.
Note that $\alpha =0$ if and only if $A=X$, 
that is, $\varphi_R$ is a Fano contraction.
There are torus invariant prime divisors $E_{\beta+1}, \cdots, E_{n+1}$ on
$X$ with $\alpha \leq \beta \leq n-1$ such
that $E_i$ is positive on $R$ for $\beta+1\leq i\leq n+1$.
Let $F$ be a general fiber of $A\to B$.
Then $F$ is a $\mathbb Q$-factorial toric Fano variety with
$\rho (F)=1$ and $\dim F=n-\beta$.
The divisors $E_{\beta+1}|_F, \cdots,
E_{n+1}|_F$ define all the torus invariant prime divisors on $F$.
In particular, $B$ is an irreducible torus invariant
subvariety of $W$ with $\dim B=\beta-\alpha$. 
When $X$ is not complete, we can reduce it to the case 
where $X$ is complete by the equivariant completion theorem 
in \cite{fujino-equiv}. For the details, see \cite{sato}. 
\end{say}

\begin{say}\label{f-say3.1.3}
We quickly review the idea of the 
proof of Theorem \ref{f-thm3.1.1} in \cite{fujino-notes}. 
We will use the same idea in the proof of Theorem \ref{f-thm3.2.1} below.  
By replacing $X$ with its projective $\mathbb Q$-factorialization, 
we may assume that $X$ is $\mathbb Q$-factorial. 
Let $R$ be an extremal ray of $\NE(X/Y)$. 
Then we consider the extremal contraction $\varphi_R:X\to W$ associated 
to $R$. If $X$ is not projective, then we can 
reduce it to the case where $X$ is projective 
by the equivariant completion theorem (see \cite{fujino-equiv}). 
By Reid's combinatorial description of $\varphi_R$, 
any fiber $F$ of $\varphi_R$ is a $\mathbb Q$-factorial projective 
toric variety with $\rho (F)=1$. 
By (sub)adjunction, we can compare $-(K_X+\Delta)\cdot C$ with 
$-K_F\cdot C$, where 
$C$ is a curve on $F$. 
So, the key ingredient of the proof of Theorem \ref{f-thm3.1.1} is the following 
proposition. 
\end{say}

\begin{prop}\label{f-prop3.1.4}
Let $X$ be a $\mathbb Q$-factorial projective 
toric $n$-fold with $\rho (X)=1$. 
Assume that $-K_X\cdot C>n$ for every 
integral curve $C$ on $X$. 
Then $X\simeq \mathbb P^n$. 
\end{prop}

For the proof, see \cite[Proposition 2.9]{fujino-notes}. 
Our proof heavily depends on the 
calculation described in \ref{f-say3.1.8} below.  

\begin{say}[Supplements to \cite{fujino-osaka}]\label{f-say3.1.5}
By the same arguments as in the proof of Proposition \ref{f-prop3.1.4}, 
we can obtain the next proposition, 
which is nothing but \cite[Proposition 2.1]{fujino-osaka}. 

\begin{prop}\label{f-prop3.1.6}
Let $X$ be a $\mathbb Q$-factorial projective 
toric $n$-fold with $\rho (X)=1$ such that $X\not\simeq \mathbb P^n$. 
Assume that $-K_X\cdot C\geq n$ for every 
integral curve $C$ on $X$. 
Then $X$ is isomorphic to the weighted projective space 
$\mathbb P(1,1,2,\cdots, 2)$. 
\end{prop}

The following proposition, which is missing in \cite{fujino-osaka}, 
may help us understand \cite{fujino-osaka}. 
This proposition says that the results in \cite{fujino-osaka} 
are compatible with \cite[Theorem 2 (a)]{fujita}. 

\begin{prop}\label{f-prop3.1.7}
Let $X$ be a projective toric $n$-fold with $n\geq 2$. 
We assume that $-K_X\equiv nD$ for some Cartier divisor $D$ on $X$ and 
$\rho(X)=1$. 
Then $D$ is very ample and $\Phi_{|D|}:X\hookrightarrow \mathbb P^{n+1}$ embeds 
$X$ into $\mathbb P^{n+1}$ as a hyperquadric. 
\end{prop}

\begin{proof}
By \cite[Theorem 3.2, Remark 3.3, and Theorem 3.4]{fujino-osaka}, 
there exists a crepant toric resolution $\varphi:Y\to X$, 
where $Y=\mathbb P_{\mathbb P^1}(\mathcal O_{\mathbb P^1}
\oplus \cdots \oplus \mathcal O_{\mathbb P^1}\oplus \mathcal O_{\mathbb P^1}(2))$ or 
$Y=\mathbb P_{\mathbb P^1}(\mathcal O_{\mathbb P^1}
\oplus \cdots\oplus 
\mathcal O_{\mathbb P^1}\oplus \mathcal O_{\mathbb P^1}(1)\oplus 
\mathcal O_{\mathbb P^1}(1))$. 
We note that $X=\mathbb P(1, 1, 2, \cdots, 2)$ when 
$Y=\mathbb P_{\mathbb P^1}(\mathcal O_{\mathbb P^1}
\oplus \cdots \oplus \mathcal O_{\mathbb P^1}\oplus \mathcal O_{\mathbb P^1}(2))$. 
We also note that $X$ is not $\mathbb Q$-factorial 
if $Y=\mathbb P_{\mathbb P^1}(\mathcal O_{\mathbb P^1}
\oplus \cdots \oplus \mathcal O_{\mathbb P^1}\oplus \mathcal O_{\mathbb P^1}(1)\oplus 
\mathcal O_{\mathbb P^1}(1))$. 
Let $\mathcal O_Y(1)$ be the tautological 
line bundle of the $\mathbb P^{n-1}$-bundle $Y$ over 
$\mathbb P^1$. 
Then we have $\mathcal O_Y(-K_Y)\simeq 
\mathcal O_Y(n)$. 
We can directly check that $\dim H^0(Y, \mathcal O_Y(1))=n+2$. 
We consider $\Phi_{|\mathcal O_Y(1)|}: Y\to\mathbb P^{n+1}$. 
By construction, 
\[
\Phi_{|\mathcal O_Y(1)|}: Y\overset{\varphi}{\longrightarrow} 
X\overset{\pi}{\longrightarrow}\mathbb P^{n+1}
\] 
is the Stein factorization of $\Phi_{|\mathcal O_Y(1)|}: Y\to \mathbb P^{n+1}$. 
By construction again, we have $\mathcal O_Y(1)\simeq 
\varphi^*\mathcal O_X(D)$. 
Since we can directly check that  
\[
\mathrm{Sym}^kH^0(Y, \mathcal O_Y(1))\to H^0(Y, \mathcal O_Y(k))
\] 
is surjective for every $k\in \mathbb Z_{>0}$, we see that 
\[
\mathrm{Sym}^kH^0(X, \mathcal O_X(D))\to H^0(X, \mathcal O_X(kD))
\] is 
also surjective for every $k\in \mathbb Z_{>0}$. 
This means that $\mathcal O_X(D)$ is very ample. 
In particular, $\pi:X\to \mathbb P^{n+1}$ is nothing but the 
embedding $\Phi_{|D|}: X\hookrightarrow \mathbb P^{n+1}$. 
Since $D^n=(\mathcal O_Y(1))^n=2$, $X$ is a hyperquadric in $\mathbb P^{n+1}$. 
\end{proof}
\end{say}

As was mentioned above, the following calculation plays 
an important role in the proof of Proposition \ref{f-prop3.1.4}. 

\begin{say}[Fake weighted projective spaces]\label{f-say3.1.8}
Now we fix $N\simeq \mathbb Z ^n$. Let $\{v_1,\cdots,v_{n+1}\}$ 
be a set of primitive vectors of $N$ such that $N_{\mathbb R}=\sum _i 
\mathbb R_{\geq 0}v_i$. 
We define $n$-dimensional cones 
\[
\sigma_i:=\langle v_1,\cdots,v_{i-1},v_{i+1},\cdots,v_{n+1}\rangle 
\] 
for $1\leq i\leq n+1$. 
Let $\Sigma$ be the complete fan generated by $n$-dimensional 
cones $\sigma_i$ and their faces for all $i$. Then 
we obtain a complete toric variety $X=X(\Sigma)$ with 
Picard number $\rho (X)=1$. 
We call it a {\em{$\mathbb Q$-factorial 
toric Fano variety with Picard number one}}. 
It is sometimes called a {\em{fake weighted projective 
space}}. 
We define $(n-1)$-dimensional cones $\mu_{i,j}=\sigma _i\cap \sigma _j$ 
for $i\ne j$. 
We can write $\sum _i a_i v_i=0$, where $a_i\in \mathbb Z_{>0}$, 
$\gcd(a_1,\cdots,a_{n+1})=1$, and $a_1\leq a_2\leq\cdots\leq a_{n+1}$ 
by changing the order. 
Then we obtain 
\[
0< V({v_{n+1}})\cdot V(\mu_{n,n+1})=\frac{\mult {(\mu_{n,n+1})}}
{\mult {(\sigma_{n})}}\leq 1, 
\]
\[
V({v_{i}})\cdot V(\mu_{n,n+1})=\frac{a_i}{a_{n+1}}\cdot
\frac{\mult {(\mu_{n,n+1})}}
{\mult {(\sigma_{n})}}, 
\]
and 
\begin{eqnarray*}
-K_{X} \cdot V(\mu_{n,n+1})&=&
\sum _{i=1}^{n+1} V({v_i})\cdot V(\mu_{n,n+1})\\
& =&
\frac {1}{a_{n+1}}
{\left(\sum_{i=1}^{n+1} a_i\right)}
\frac{\mult {(\mu_{n,n+1})}}
{\mult {(\sigma_{n})}}\leq n+1. 
\end{eqnarray*} 
We note that 
\[
\frac{\mult(\sigma_n)}{\mult(\mu_{n, n+1})}\in \mathbb Z_{>0}. 
\]
For the procedure to compute intersection numbers, 
see \ref{f-say2.1.5} or \cite[p.100]{fulton}. 
If $-K_{X} \cdot V(\mu_{n,n+1})=n+1$, then $a_i=1$ for every 
$i$ and $\mult (\mu_{n,n+1})=\mult (\sigma_{n})$. 

We note that the above calculation plays 
crucial roles in \cite{fujino-notes}, \cite{fujino-osaka}, 
\cite{fujino-ishitsuka}, and this paper (see the 
proof of Theorem \ref{f-thm3.2.1}, and so on). 

\begin{lem}\label{f-lem3.1.9} 
We use the same notation as in {\em{\ref{f-say3.1.8}}}. 
We consider the sublattice 
$N'$ of $N$ spanned by 
$\{v_1, \cdots, v_{n+1}\}$. 
Then the natural inclusion 
$N'\to N$ induces a finite 
toric morphism $f:X'\to X$ from a 
weighted projective space $X'$ such 
that $f$ is {}\'etale in codimension one. 
In particular, $X(\Sigma)$ is a weighted projective 
space if and only if $\{v_1, \cdots, v_{n+1}\}$ 
generates $N$. 
\end{lem}

We note the above elementary lemma. 
Example \ref{f-ex3.1.10} shows that there are 
many fake weighted projective spaces which are not weighted projective 
spaces. 

\begin{ex}\label{f-ex3.1.10}
We put $N=\mathbb Z^3$. Let $\{e_1, e_2, e_3\}$ be the standard 
basis of $N$. 
We put $v_1=e_1$, $v_2=e_2$, $v_3=e_3$, 
and $v_4=-e_1-e_2-e_3$. 
The fan $\Sigma$ is the subdivision of $N_{\mathbb R}$ 
by $\{v_1, v_2, v_3, v_4\}$. 
Then $Y=X(\Sigma)\simeq \mathbb P^3$. 
We consider a new lattice \[N^\dag
=N+\left(\frac{1}{2}, \frac{1}{2}, 0\right)\mathbb Z.\]  
The natural inclusion $N\to N^\dag$ induces a finite toric morphism $Y\to 
X$, which is {}\'etale in codimension one. It is easy to see that 
$K_X$ is Cartier and $-K_X\sim 4D_4$, where 
$D_4=V(v_4)$ is not Cartier but $2D_4$ is Cartier. 
Since $\{v_1, v_2, v_3, v_4\}$ does not span the lattice $N^\dag$, $X$ is not a 
weighted projective space. Of course, $X$ 
is a fake weighted projective space. 
\end{ex}
\end{say}

\subsection{New estimate of lengths of extremal rays}\label{f-subsec3.2} 
The following theorem is one of the main theorems of this paper, in which 
we prove new estimates of $K_X$-negative 
extremal rays of birational type. 
We will see that they are the best by Examples \ref{f-ex3.3.1} 
and \ref{f-ex3.3.2}. 

\begin{thm}[Lengths of extremal rays of birational type, Theorem \ref{f-thm1.2}]
\label{f-thm3.2.1}
Let $f:X\to Y$ be a projective toric morphism with $\dim X=n$. 
Assume that $K_X$ is $\mathbb Q$-Cartier. 
Let $R$ be a $K_X$-negative 
extremal ray of $\NE(X/Y)$ and let $\varphi_R:X\to W$ 
be the contraction morphism associated to $R$. We put 
\[l(R)=\min_{[C]\in R} (-K_X\cdot C).\] 
and call it the length of $R$. 
Assume that $\varphi_R$ is birational. 
Then we obtain 
\[
l(R)<d+1, 
\]
where \[d=\max_{w\in W} \dim \varphi^{-1}_R(w)\leq n-1. \] 
When $d=n-1$, we have a sharper inequality 
\[
l(R)\leq d=n-1. 
\]
In particular, if 
$l(R)=n-1$, 
then $\varphi_R:X\to W$ can be described as follows. 
There exists a torus invariant smooth point $P\in W$ such that 
$\varphi_R:X\to W$ 
is a weighted blow-up at $P$ with the weight $(1, a, \cdots, a)$ for some 
positive integer $a$. 
In this case, the exceptional locus $E$ of $\varphi_R$ is a torus invariant 
prime divisor and is isomorphic to $\mathbb P^{n-1}$. 
Moreover, $X$ is $\mathbb Q$-factorial in a neighborhood of 
$E$. 
\end{thm}

The idea of the proof of Theorem \ref{f-thm3.2.1} is the same as 
that of Theorem \ref{f-thm3.1.1}. 

\begin{proof}[Proof of Theorem \ref{f-thm3.2.1}] 
In Step \ref{f-3.2.1step1}, we will explain how to 
reduce problems to the case where $X$ is $\mathbb Q$-factorial. 
Then we will prove the inequality $l(R)<d+1$ in Step 
\ref{f-3.2.1step2}. 
In Step \ref{f-3.2.1step3}, we will treat the case where $X$ is $\mathbb Q$-factorial 
and $l(R)\geq n-1$. 
Finally, in Step \ref{f-3.2.1step4}, we will treat the case where 
$l(R)\geq n-1$ under the assumption that $X$ is not necessarily 
$\mathbb Q$-factorial. 

\begin{step}\label{f-3.2.1step1} 
In this step, we will explain how to reduce problems to the case where $X$ is 
$\mathbb Q$-factorial. 

Without loss of generality, we may assume that $W=Y$. 
Let $\pi:\widetilde X\to X$ be a small projective 
$\mathbb Q$-factorialization (see, for example, 
\cite[Corollary 5.9]{fujino-notes}). 
Then we can take an extremal ray $\widetilde R$ of $\NE(\widetilde X/W)$ 
and construct the following commutative diagram 
\[
\xymatrix{
\widetilde X \ar[r]^{\varphi_{\widetilde R}} \ar[d]_\pi& \widetilde W \ar[d]\\ 
X \ar[r]_{\varphi_R} & W
}
\]
where $\varphi_{\widetilde R}$ is the contraction morphism 
associated to $\widetilde R$. 
We note that $\varphi_{\widetilde R}$ must be small 
when $\varphi_R$ is small. 
We write 
\[
\begin{matrix}
 & \widetilde A & \longrightarrow &\widetilde B \\
                   & \cap & &\cap\\
\varphi_{\widetilde R}: & \widetilde X& \longrightarrow &\widetilde W, \\
\end{matrix}
\] 
where $\widetilde A$ is the exceptional locus of 
$\varphi_{\widetilde R}$ and $\widetilde B$ is the image of $\widetilde A$. 
Let $\widetilde F$ be a general fiber of $\widetilde A\to \widetilde B$. 
Then $\widetilde F$ is a fake weighted projective 
space as in \ref{f-say3.1.2}, that is, $\widetilde F$ is a 
$\mathbb Q$-factorial 
toric Fano variety with Picard number one. 
Since $\rho(\widetilde F)=1$, $\pi:\widetilde F\to F:=\pi(\widetilde F)$ is 
finite. 
Therefore, by definition, $\dim \widetilde F=\dim F\leq d$ since 
$\varphi_R(F)$ is a point. 
Let $\widetilde C$ be a curve in $\widetilde F$ and let $C$ be the 
image of $\widetilde C$ by $\pi$ with the reduced scheme structure. 
Then we obtain 
\[
-K_{\widetilde X}\cdot \widetilde C=-\pi^*K_X\cdot \widetilde C=-mK_X\cdot 
C,  
\] 
where $m$ is the mapping degree of $\widetilde C\to C$. 
Thus, if $-K_{\widetilde X}\cdot \widetilde C$ satisfies the 
desired inequality, then $-K_X
\cdot C$ also satisfies the same inequality. 
Therefore, for the proof of $l(R)<d+1$, 
we may assume that $X$ is $\mathbb Q$-factorial and $W=Y$ 
by replacing $X$ and $Y$ with $\widetilde X$ and $\widetilde W$, respectively. 
\end{step}
\begin{step}\label{f-3.2.1step2}
In this step, we will prove the desired inequality $l(R)<d+1$ 
under the assumption that $X$ is $\mathbb Q$-factorial.  

We write 
\[
\begin{matrix}
 & A & \longrightarrow &B \\
                   & \cap & &\cap\\
\varphi_R: & X& \longrightarrow &W, \\
\end{matrix}
\] 
where $A$ is the exceptional locus of $\varphi_R$ and $B$ is the 
image of $A$. It is well known that $A$ is irreducible. 
We put $\dim A=n-\alpha$ and $\dim B=\beta-\alpha$ as in \ref{f-say3.1.2}. 
Let $F$ be a general fiber of $A\to B$. 
Then, it is known that $F$ is a $\mathbb Q$-factorial toric Fano variety with 
Picard number one. It is also well known that there exist torus invariant prime divisors 
$E_1,\cdots, E_{\alpha}$ on $X$ such that $E_i$ is negative 
on $R$ for every $i$ and 
$A$ is $E_1\cap \cdots \cap E_{\alpha}$. 
By (sub)adjunction, we have 
\[
(K_X+E_1+\cdots +E_\alpha)|_A=K_A+D
\] 
for some effective $\mathbb Q$-divisor $D$ on $A$. 
Note that $D$ is usually called a {\em{different}}. 
Let $C$ be a curve in $F$. 
Then 
\begin{equation}\label{f-eq3.1}
\begin{split}
-K_X\cdot C&=-(K_A+D)\cdot C+E_1\cdot C+\cdots 
+E_{\alpha}\cdot C\\&<-(K_A+D)\cdot C\leq -K_F\cdot C. 
\end{split}
\end{equation}
By \cite[Proposition 2.9]{fujino-notes} (see also \ref{f-say3.1.8}), 
there exists a torus invariant curve $C$ on $F$ such that 
$-K_F\cdot C\leq \dim F+1=n-\beta+1$. 
Therefore, we obtain \[-K_X\cdot C<n-\beta+1
=d+1\leq
n\] since $\beta\geq \alpha \geq 1$. 
This means that $l(R)<d+1$. 
By combining it with Step \ref{f-3.2.1step1}, 
we have $l(R)<d+1$ without assuming that $X$ is $\mathbb Q$-factorial. 
 
We close this step with easy useful remarks. 

\begin{rem}\label{f-rem3.2.2}
We note that if $F\not\simeq \mathbb P^{n-\beta}$ 
in the above argument, then we can choose $C$ such that 
$
-K_F\cdot C\leq \dim F=n-\beta
$ 
(see Theorem \ref{f-thm3.1.1}). 
\end{rem}

\begin{rem}\label{f-rem3.2.3}
If $X$ is Gorenstein, then $-K_X\cdot C<n$
implies $-K_X\cdot C\leq n-1$. 
Therefore, by combining it with 
Step \ref{f-3.2.1step1}, we can easily see that 
the estimate $l(R)\leq n-1$ always holds for Gorenstein 
(not necessarily $\mathbb Q$-factorial) toric varieties. 

If $\varphi_R$ is small, then we can find $C$ such that 
$-K_X\cdot C<n-1$ and $[C]\in R$ since 
we know $\beta\geq \alpha \geq 2$. 
Therefore, by combining it with Step \ref{f-3.2.1step1}, the estimate 
$l(R)<n-1$ always holds for 
(not necessarily $\mathbb Q$-factorial) toric 
varieties, when $\varphi_R$ is small. 
\end{rem}
\end{step}

\begin{step}\label{f-3.2.1step3} %%%%%%%%%%%%%%%%%%%%%%%%%%
In this step, we will investigate the case where 
$l(R)\geq n-1$ under the assumption that $X$ is $\mathbb Q$-factorial. 

We will use the same notation as in Step \ref{f-3.2.1step2}. 
In this case, we see that $-K_X\cdot C\geq n-1$ for every curve 
$C$ on $F$. Then, we see that 
$\dim A=\dim F=n-1$, $F\simeq \mathbb P^{n-1}$ and $\dim B=0$ 
(see Remark \ref{f-rem3.2.2}). 
Equivalently, $\varphi_R$ contracts $F\simeq \mathbb P^{n-1}$ to a torus invariant point 
$P\in W$. Let $\langle e_1, \cdots, e_n\rangle$ be the 
$n$-dimensional cone corresponding to $P\in W$. 
Then $X$ is obtained by the star subdivision of $\langle e_1, \cdots, e_n\rangle$ 
by $e_{n+1}$, where $be_{n+1}=a_1e_1+\cdots +a_ne_n$, $b\in \mathbb Z_{>0}$ and 
$a_i\in \mathbb Z_{>0}$ for all $i$. 
We may assume that $\gcd (b, a_1,\cdots, a_n)=1$, 
$\gcd (b, a_1, \cdots, a_{i-1}, a_{i+1}, \cdots a_n)=1$ for all $i$, 
and $\gcd (a_1, \cdots, a_n)=1$. Without loss of generality, we may assume that 
$a_1\leq \cdots \leq a_n$ by changing the order. 
We write $\sigma_i=\langle e_1, \cdots, e_{i-1}, e_{i+1}, \cdots, 
e_{n+1}\rangle$ for all $i$ and $\mu_{k,l}=\sigma_k\cap \sigma_l$ for 
$k\ne l$. Then 
\begin{equation}\label{f-eq3.2}
-K_X\cdot V(\mu _{k,n})=\frac{1}{a_n}
\left(\sum _{i=1}^{n}
a_i-b\right)\frac{\mult(\mu_{k,n})}{\mult(\sigma_k)}\geq n-1
\end{equation} for 
$1\leq k\leq n-1$. Then $\mult (\mu_{k,n})=\mult (\sigma_k)$ for 
$1\leq k\leq n-1$. Thus, $a_k$ divides $a_n$ for $1\leq k\leq n-1$. 

\begin{case}\label{f-case1}If $a_1=a_n$, then $a_1=\cdots =a_n=1$. 
In this case $-K_X\cdot V(\mu _{k,n})\geq 
n-1$ implies $b=1$. And we have $\mult (\mu _{k,l})=\mult 
(\sigma _k)$ for $1\leq k\leq n$, $1\leq l\leq n$, and $k\ne l$. In particular, 
$\mult (\sigma _1) =\mult (\mu_{1, l})$ for $2\leq l\leq n$. 
This implies $\mult (\sigma _1)=1$. Since $e_{n+1}=e_1+\cdots 
+e_n$, $\langle e_1, \cdots, e_n\rangle$ is a nonsingular cone. 
Therefore, $\varphi_R:X\to W$ is a blow-up at a smooth point $P$. 
Of course, $l(R)=n-1$. 
\end{case}

\begin{case}\label{f-case2}
Assume that $a_1\ne a_n$. If $a_2\ne a_n$, 
then $\frac{a_1}{a_n}\leq \frac{1}{2}$ and  
$\frac{a_2}{a_n}\leq \frac{1}{2}$. This contradicts $-K_X\cdot 
V(\mu_{k,l})\geq n-1$. 
Therefore, $a_1=1$ and $a_2=\cdots =a_n=a$ for some positive integer 
$a\geq 2$. 
The condition $-K_X\cdot V(\mu _{k,n})\geq n-1$ implies $b=1$. 
Thus, $\mult (\mu_{k,l})=\mult (\sigma_k)$ for $1\leq 
k\leq n$, $2\leq l\leq n$, and $k\ne l$. In particular, 
$\mult (\sigma _1)=\mult (\mu _{1, l})$ for 
$2\leq l\leq n$. 
Thus, $\mult (\sigma _1)=1$. 
Since \[e_{n+1}=e_1+ae_2+\cdots +ae_n, \] 
$\langle e_1, \cdots, e_n\rangle$ is a nonsingular cone. 
Therefore, $\varphi_R:X\to W$ is a weighted blow-up at a smooth point $P\in W$ with 
the weight $(1, a, \cdots, a)$. 
In this case, $K_X=\varphi^*_RK_W+(n-1)aE$, where 
$E\simeq \mathbb P^{n-1}$ is the exceptional divisor 
and $l(R)=n-1$ (see Proposition \ref{f-prop3.2.6} below). 
\end{case}
Anyway, when $X$ is $\mathbb Q$-factorial, 
we obtain that $l(R)\geq n-1$ implies $l(R)=n-1$. 
Therefore, the estimate $l(R)\leq n-1$ always holds when $X$ is $\mathbb Q$-factorial 
and $\varphi_R$ is birational. 
\end{step}
\begin{step}\label{f-3.2.1step4}
In this final step, we will treat the case where $l(R)\geq n-1$ under the 
assumption that $X$ is not necessarily $\mathbb Q$-factorial. 

Let $\pi:\widetilde X\to X$ be a small projective 
$\mathbb Q$-factorialization as in Step \ref{f-3.2.1step1}. 
By the argument in Step \ref{f-3.2.1step1}, 
we can find a $K_{\widetilde X}$-negative extremal ray $\widetilde R$ 
of $\NE (\widetilde X/W)$ such that 
$l(\widetilde R)\geq n-1$. Therefore, by Step \ref{f-3.2.1step3}, 
the associated contraction 
morphism $\varphi_{\widetilde R}:\widetilde X\to \widetilde W$ is a weighted 
blow-up at a smooth point $\widetilde P\in \widetilde W$ with 
the weight $(1, a, \cdots, a)$ for some positive integer $a$. 
Let $\widetilde E$ $(\simeq \mathbb P^{n-1})$ be the 
$\varphi_{\widetilde R}$-exceptional 
divisor on $\widetilde X$. We put $E=\pi(\widetilde E)$. 
Then it is easy to see that $E\simeq \mathbb P^{n-1}$ and that 
$\pi:\widetilde E\to E$ is an isomorphism. 
\begin{lem}\label{f-lem3.2.4} 
$\pi: \widetilde X\to X$ is an isomorphism over some open neighborhood of $E$. 
\end{lem}

\begin{proof}[Proof of Lemma \ref{f-lem3.2.4}]
We will get a contradiction by assuming that $\pi:\widetilde X\to X$ is not 
an isomorphism over any open neighborhood of $E$. 
Since $\varphi_{\widetilde R}$ is a weighted blow-up as 
described in the case where $X$ is $\mathbb Q$-factorial 
(see Step \ref{f-3.2.1step3}) and 
$\pi$ is a crepant small toric morphism by construction, 
the fan of $\widetilde{X}$ contains $n$-dimensional cones 
\[
\sigma_i:=\langle\{e_1,\ldots,e_{n+1}\}\setminus\{e_i\}\rangle,
\] 
for $1\le i\le n$, 
where $\{e_1,\ldots,e_n\}$ is the standard basis of $N=\mathbb Z^n$ 
and 
$e_{n+1}:=e_1+ae_2+\cdots+ae_n$ with $a\in\mathbb{Z}_{>0}$. 
Since we assume that $\pi:\widetilde 
X\to X$ is not an isomorphism over any open 
neighborhood of $E$, there exists at least one non-simplicial 
$n$-dimensional cone 
$\sigma$ in the fan of $X$ 
such that  $\sigma$ contains one of the above $n$-dimensional cones. 
By symmetry, it is sufficient to consider the two cases where 
$\sigma$ contains $\sigma_n$ or $\sigma_1$. 

First, we suppose $\sigma_n\subset\sigma$. Let $x=x_1e_1+\cdots +x_ne_n
\in N$ be the primitive generator 
for some one-dimensional 
face of $\sigma$ which is not contained in $\sigma_n$. 
Then, by considering the facets of $\sigma_n$, we have the inequalities 
$ax_1-x_n\ge 0$, $x_i-x_n\ge 0$ for $2\le i\le n-1$, and $x_n<0$. 
If $x_1-x_n<0$, then $x_1<x_n<0$. 
This means that $ax_1-x_n\leq x_1-x_n<0$. 
This is a contradiction. Therefore, the inequality $x_1-x_n\geq 0$ also holds. 

\begin{claim}\label{f-claim}
$x_i\leq 0$ for every $i\ne n$. 
\end{claim}

\begin{proof}[Proof of Claim] 
Suppose $x_i>0$ for some $i\ne n$. 
We note that $x$ must be contained 
in the hyperplane passing through 
the points $e_1,\ldots,e_{n-1},e_{n+1}$ since 
$\pi$ is crepant, that is, $K_{\widetilde X}=
\pi^*K_X$. So the equality 
\begin{equation*}
\begin{split}
1&=x_1+\cdots+x_{n-1}-(n-2)x_n
\\&=(x_1-x_n)+\cdots+(x_{i-1}-x_n)+x_i+(x_{i+1}-x_n)+\cdots+(x_{n-1}-x_n)
\end{split}
\end{equation*}
holds. Therefore, $x_j-x_n=0$ must hold for every $j\neq i$, and 
$x_i=1$. If $i\ne 1$, then we have $a=1$ since $ax_1-x_n=(a-1)x_n\ge 0$ and 
$x_n<0$. 
However, the linear relation 
\[x+(-x_n)e_{n+1}=(1-x_n)e_i\] means that $\pi$ 
contracts a divisor $V(e_i)$. 
This is a contradiction because $\pi$ is small by construction. 
If $i=1$, then we have $ax_1-x_n=a-x_n>0$ since $a>0$ and $-x_n>0$. 
However, the linear relation 
\[
ax+(-x_n)e_{n+1}=(a-x_n)e_1
\] 
means that $\pi$ contracts a divisor $V(e_1)$. This is a 
contradiction because $\pi$ is small by construction. 
In any case, we obtain that $x_i\leq 0$ holds for $1\le i \le n-1$.  
\end{proof}
Therefore, the linear relation 
\[
(-x_1)e_1+\cdots+(-x_{n})e_{n}+x=0
\]
says that the cone $\langle e_1,\ldots,e_n,x\rangle$ 
contains a positive dimensional linear subspace of 
$N_{\mathbb R}$ because 
$-x_i\ge 0$ for $1\le i\le n-1$ and $-x_n>0$. 
Since $\langle e_1,\ldots,e_n,x\rangle$ must be contained in a strongly 
convex cone in the fan of $W$, this is a contradiction. 

Next, we suppose $\sigma_1\subset\sigma$. 
We can apply the same argument as above. 
Let $x=x_1e_1+\cdots +x_ne_n
\in N$ be the primitive generator 
for some one-dimensional 
face of $\sigma$ which is not contained in $\sigma_1$. 
In this case, we can obtain the 
inequalities $x_i-ax_1\ge 0$ for $2\le i\le n$, and $x_1<0$ by considering 
the facets of $\sigma_1$, and  
the equality $(1-(n-1)a)x_1+x_2+\cdots+x_n=1$ by the fact that 
$\pi$ is crepant. 
If $x_i>0$ for some $2\le i\le n$, then the equality
\begin{equation*}
\begin{split}
1&=(1-(n-1)a)x_1+x_2+\cdots+x_n
\\&=(1-a)x_1+(x_2-ax_1)+\cdots+(x_{i-1}-ax_1)+x_i+(x_{i+1}-ax_1)
+\cdots+(x_{n}-ax_1)
\end{split}
\end{equation*}
tells us that $a=1$ because $x_1<0$, and that 
$x_j-x_1=0$ for every $j\neq i$ and $x_i=1$. 
Therefore, as in the proof of Claim, 
we get a contradiction by the linear relation 
\[
x+(-x_1)e_{n+1}=(1-x_1)e_i. 
\] 
So we obtain that $x_i\leq 0$ holds for 
$2\leq i\leq n$. 
Thus we get a linear relation 
\[
(-x_1)e_1+\cdots+(-x_{n})e_{n}+x=0
\]
as above, where $-x_i\geq 0$ for $2\leq i\leq n$ and $-x_1>0$. 
This means that the cone $\langle e_1, \cdots, e_n, x\rangle$ contains 
a positive dimensional linear subspace of $N_{\mathbb R}$. 
This is a contradiction as explained above. 
 
In any case, we get a contradiction. Therefore, 
$\pi:\widetilde X\to X$ is an isomorphism 
over some open neighborhood of $E$. 
\end{proof} 
Since $\pi:\widetilde X\to X$ is an isomorphism over some 
open neighborhood of $E$ by Lemma \ref{f-lem3.2.4}, 
we see that $E$ is $\mathbb Q$-Cartier. 
Therefore, the exceptional locus of $\varphi_R$ coincides with 
$E\simeq \mathbb P^{n-1}$. Thus $\varphi_R:X\to W$ is 
a weighted blow-up 
at a torus invariant smooth 
point $P\in W$ with the weight $(1, a, \cdots, a)$ for some positive integer $a$. 
\end{step}
So, we complete the proof of Theorem \ref{f-thm3.2.1}. 
\end{proof}

\begin{rem}\label{f-rem3.2.5} 
If $B$ is complete, then we can make $C$ a torus invariant curve on $X$ in 
Theorem \ref{f-thm3.2.1}. 
For the details, see the proof of \cite[Theorem 0.1]{fujino-notes}. 
\end{rem}

We explicitly state the basic properties 
of the weighted blow-up in Theorem 
\ref{f-thm3.2.1} for the reader's convenience. 

\begin{prop}\label{f-prop3.2.6}
Let $\varphi:X\to \mathbb A^n$ be the weighted blow-up at  $0\in 
\mathbb A^n$ with the weight $(1, a, \cdots, a)$ for some 
positive integer $a$. 
If $a=1$, then $\varphi$ is the standard blow-up at $0$. 
In particular, $X$ is smooth. 
If $a\geq 2$, then $X$ has only canonical Gorenstein singularities which are not 
terminal singularities. Furthermore, the exceptional locus $E$ of $\varphi$ 
is isomorphic to 
$\mathbb P(1, a, \cdots, a)\simeq \mathbb P^{n-1}$ and 
\[K_X=\varphi ^*K_{\mathbb A^n}+(n-1)aE. \] 
We note that $E$ is not Cartier on $X$ if $a\ne 1$. 
However, $aE$ is a Cartier divisor on $X$. 
\end{prop}

\begin{proof}
We can check the statements by direct calculation. 
\end{proof}

Let us see an important related example. 

\begin{ex}\label{f-ex3.2.7}
We fix $N=\mathbb Z^n$ and let $\{e_1, \cdots, e_n\}$ be the
standard basis of $N$. We consider the cone $\sigma =\langle e_1,
\cdots, e_n\rangle$ in $N'=N+\mathbb Z e_{n+1}$, where $e_{n+1}=
\frac{1}{b}(1, a, \cdots, a)$.
Here, $a$ and $b$ are positive integers such that $\gcd(a,b)=1$.
We put $Y=X(\sigma)$ is the associated affine toric variety which has only one
singular point $P$.
We take a weighted blow-up of $Y$ at $P$ with
the weight $\frac{1}{b}(1, a, \cdots, a)$.
This means that we divide $\sigma$ by $e_{n+1}$ 
and obtain a fan $\Sigma$ of $N'_{\mathbb R}$.
We define $X=X(\Sigma)$. Then the induced toric projective birational
morphism $f:X\to Y$ is the desired weighted blow-up. It is obvious that $X$
is $\mathbb Q$-factorial and $\rho (X/Y)=1$. We can easily
obtain \[K_X=f^*K_Y+\left(\frac{1+(n-1)a}{b}-1\right)E, \] where $E=V(e_{n+1})\simeq
\mathbb P^{n-1}$ is the exceptional divisor of $f$, 
and \[-K_X\cdot C=(n-1)-\frac{b-1}{a},\]  
where $C=V(\langle e_2, \cdots, e_{n-1}, e_{n+1}\rangle)$ is
a torus invariant irreducible curve on $E$. 
We note that \[-(K_X+\delta E)\cdot C>n-1\] 
if and only if 
\[
\delta>\frac{b-1}{b}
\] 
since $E\cdot C=-\frac{b}{a}$. 
\end{ex}

In subsection \ref{f-subsec3.3}, we will see more sophisticated examples 
(see Examples \ref{f-ex3.3.1} and \ref{f-ex3.3.2}), which show the 
estimates obtained in Theorem \ref{f-thm3.2.1} are the best. 

By the proof of Theorem \ref{f-thm3.2.1}, we can prove the following 
theorem. 

\begin{thm}\label{f-thm3.2.8}
Let $f:X\to Y$ be a projective toric morphism with $\dim X=n$ and
let $\Delta=\sum \delta_i \Delta_i$ be an effective $\mathbb R$-divisor on $X$,
where $\Delta_i$ is a torus invariant prime divisor and $0\leq \delta_i\leq 1$ for
every $i$.
Let $R$ be an extremal ray of $\NE(X/Y)$ and let $\varphi _R: X\to W$
be the extremal contraction morphism associated to $R$.
Assume that $X$ is $\mathbb Q$-factorial and $\varphi_R$ is
birational.
If \[\min _{[C]\in R}(-(K_X+\Delta)\cdot C)>n-1, \] 
then
$\varphi_R:X\to W$ is the weighted blow-up described in Example \ref{f-ex3.2.7} 
and $\xSupp \Delta\supset E$, where $E\simeq \mathbb P^{n-1}$ is the
exceptional divisor of $\varphi_R$.
\end{thm}

\begin{proof}
We use the same notation as in Step \ref{f-3.2.1step2} in the proof of 
Theorem \ref{f-thm3.2.1}. 
Since \[(E_1+\cdots +E_\alpha-\Delta)\cdot C\leq 0,\] 
we obtain 
\begin{equation*}
\begin{split}
-(K_X+\Delta)\cdot C &= -(K_A+D)\cdot C +
(E_1+\cdots +E_\alpha-\Delta)\cdot C\\ 
&\leq -(K_A+D)\cdot C 
\\&\leq -K_F\cdot C
\end{split} 
\end{equation*} 
(see \eqref{f-eq3.1}). 
By assumption, 
$-(K_X+\Delta)\cdot C>n-1$. 
This implies that 
$n-1<-K_F\cdot C$. 
Therefore, we obtain $\dim A=\dim F=n-1$, $F\simeq \mathbb P^{n-1}$ and 
$\dim B=0$. 
In this situation, 
\[
-(K_X+\Delta)\cdot C\leq -(K_X+A)\cdot C
\] 
always holds. Thus we have 
\begin{equation*}
-(K_X+A)\cdot V(\mu _{k,n})=\frac{1}{a_n}
\left(\sum _{i=1}^{n}
a_i\right)\frac{\mult(\mu_{k,n})}{\mult(\sigma_k)}> n-1
\end{equation*} for 
$1\leq k\leq n-1$ (see \eqref{f-eq3.2}). 
We note that $A=V(e_{n+1})$. 
Thus, by the same arguments as in the proof of 
Theorem \ref{f-thm3.2.1}, 
we see that $\varphi_R$ is the weighted blow-up 
described in Example \ref{f-ex3.2.7}. 
More precisely, we obtain that 
$(a_1, \cdots, a_n)=(1, \cdots, 1)$ or $(1, a, \cdots, a)$ and 
that $\sigma_1$ is a nonsingular cone. 
However, $b$ is not necessarily $1$ in the proof of Theorem \ref{f-thm3.2.1}.   
By direct calculation, 
we have 
$\xSupp \Delta\supset E$, where 
$E(=A=F)$ is the exceptional divisor of $\varphi_R$.  
\end{proof}

Finally, we prove 
the following theorem. 

\begin{thm}[Theorem \ref{f-thm1.3}]\label{f-thm3.2.9} 
Let $X$ be a $\mathbb Q$-Gorenstein 
projective toric $n$-fold with $\rho (X)\geq 2$. 
Let $R$ be a $K_X$-negative extremal ray of 
$\NE(X)$ such that 
\[
l(R)=\min _{[C]\in R}(-K_X\cdot C)>n-1. 
\] 
Then the extremal contraction $\varphi_R:X\to W$ associated to 
$R$ is a $\mathbb P^{n-1}$-bundle over $\mathbb P^1$. 
\end{thm}
\begin{proof}
We divide the proof into several steps. 
From Step \ref{f-3.2.9step1} to Step \ref{f-3.2.9step4}, 
we will prove this theorem under the extra assumption that $X$ is 
$\mathbb Q$-factorial. In Step \ref{f-3.2.9step5}, we will 
prove that $X$ is always $\mathbb Q$-factorial if there exists an 
extremal ray $R$ with $l(R)>n-1$. 
\setcounter{step}{0}
\begin{step}\label{f-3.2.9step1}
We consider the contraction morphism 
$\varphi_R:X\to W$ associated to $R$. 
By Theorem \ref{f-thm3.2.1}, 
$\varphi_R$ is a Fano contraction, that is, 
$\dim W<\dim X$. 
Let $F$ be a general fiber of $\varphi_R$ and let $C$ be 
a curve on $F$. 
Then, by adjunction, we have 
\[
-K_X\cdot C=-K_F\cdot C. 
\]
We note that $F$ is a fake weighted projective 
space. By Theorem \ref{f-thm3.1.1}, 
$F\simeq \mathbb P^{n-1}$, $W=\mathbb P^1$, and 
$\rho(X)=2$. 
\end{step}
\begin{step}\label{f-3.2.9step2} 
Without loss of generality, $\varphi_R:X\to W$ is induced by 
the projection $\pi:N=\mathbb Z^n\to \mathbb Z, 
(x_1, \cdots, x_n)\mapsto x_n$. 
We put 
\begin{align*}
v_1  &= (1,0,\cdots, 0), & v_2&=(0,1,0, \cdots, 0), & &\quad \quad  \cdots , \\ 
v_{n-1}& =(0,\cdots, 0, 1, 0), & v_n  &= (-1,\cdots, -1, 0), 
&v_+&=(b_1,\cdots, b_{n-1}, a_+), \\ v_-&=(c_1, \cdots, c_{n-1}, -a_-), &
\end{align*} 
where $a_+$ and $a_-$ are positive integers. 
More precisely, $v_i$ denotes the vector 
with a $1$ in the $i$th coordinate and $0$'s elsewhere for 
$1\leq i\leq n-1$. 
We may assume that the fan $\Sigma$ corresponding to the toric 
variety $X$ is the subdivision of $N_{\mathbb R}$ by $v_1, \cdots, v_n, 
v_+$, and $v_-$. We note that the following equalities 
\begin{equation}\label{f-ex3.3}
\begin{cases}
D_1-D_n+b_1D_++c_1D_-=0\\
D_2-D_n+b_2D_++c_2D_-=0\\
\quad \quad \quad \quad \vdots\\
D_{n-1}-D_n+b_{n-1}D_++c_{n-1}D_-=0\\
a_+D_+-a_-D_-=0
\end{cases}
\end{equation}
hold, where $D_i=V(v_i)$ for 
every $i$ and $D_{\pm}=V(v_{\pm})$. 
We note that it is sufficient to prove that $a_+=a_-=1$. 
\end{step}
\begin{step}\label{f-3.2.9step3}
In this step, we will prove that $a_+=1$ holds. 

By taking a suitable coordinate change, we may assume that 
\[
0\leq b_1, \cdots, b_{n-1}<a_+ 
\] holds. 
If $b_i=0$ for every $i$, then $a_+=1$ since 
$v_+$ is a primitive vector of $N$. 
From now on, we assume that $b_{i_0}\ne 0$ for some $i_0$. 
Without loss of generality, we may assume that 
$b_1\ne 0$. 
We put 
\[
C=V(\langle v_2, \cdots, v_{n-1}, v_+\rangle). 
\]
Then $C$ is a torus invariant curve contained in a fiber of $\varphi_R:X\to W$. 
We have 
\[
D_1\cdot C=\frac {\mult (\langle v_2, \cdots, v_{n-1}, v_+\rangle)}
{\mult (\langle v_1, \cdots, v_{n-1}, v_+\rangle)}=\frac 
{\gcd(a_+, b_1)}{a_+}
\] (see \ref{f-say2.1.5})
and $
D_+\cdot C=D_-\cdot C=0
$.  
We note that 
$D_i\cdot C=D_1\cdot C$ for every $i$ 
by \eqref{f-ex3.3}. 
Therefore, 
we obtain 
\[
-K_X\cdot C=\frac{n\gcd (a_+, b_1)}{a_+}. 
\] 
Since $0<b_1<a_+$, we see $\gcd(a_+, b_1)<a_+$. 
Thus, the following inequality 
\[
\frac{\gcd(a_+, b_1)}{a_+}\leq \frac{1}{2}
\]
holds. 
This means that 
\[
-K_X\cdot C\leq \frac{n}{2}\leq n-1. 
\] 
This is a contradiction. 
Therefore, $b_i=0$ for every $i$ and $a_+=1$. 
\end{step}

\begin{step}\label{f-3.2.9step4} 
By the same argument, we get $a_-=1$. 
Thus, we see that $\varphi_R:X\to W$ is a 
$\mathbb P^{n-1}$-bundle over $\mathbb P^1$. 
\end{step}
\begin{step}\label{f-3.2.9step5}
In this step, we will prove that $X$ is $\mathbb Q$-factorial. 

We assume that $X$ is not $\mathbb Q$-factorial. Let $\pi:\widetilde 
X\to X$ be a small projective $\mathbb Q$-factorialization. 
We note that $\rho(\widetilde X)>\rho (X)\geq 2$ since 
$X$ is not $\mathbb Q$-factorial. 
By the argument in Step \ref{f-3.2.1step1} in the proof of Theorem \ref{f-thm3.2.1}, 
there exists an extremal ray $\widetilde R$ of $\NE(\widetilde X/W)$ 
with $l(\widetilde R)>n-1$. 
Let $\varphi_{\widetilde R}: \widetilde X\to \widetilde W$ be the contraction morphism 
associated to $\widetilde R$. Then, by the argument in Step \ref{f-3.2.9step1}, 
we see that $\rho(\widetilde X)=2$ and that $\varphi_{\widetilde R}$ is 
nothing but $\varphi_R\circ \pi$. 
This is a contradiction because $\rho(\widetilde X)>2$. 
This means that $X$ is always $\mathbb Q$-factorial. 
\end{step} 
Therefore, we get the desired statement. 
\end{proof}

We close this subsection with an easy 
example, which shows that Theorem \ref{f-thm3.2.9} is sharp. 

\begin{ex}\label{f-ex3.2.10}
We consider $N=\mathbb Z^2$, 
$
v_1=(0, 1), v_2=(0, -1), v_+=(2, 1), v_-=(-1, 0)
$ 
and the projection $\pi:N=\mathbb Z^2\to \mathbb Z, (x_1, x_2)\mapsto x_1$. 
Let $\Sigma$ be the fan obtained by subdividing $N_\mathbb R$ by 
$\{v_1, v_2, v_+, v_-\}$. 
Then $X=X(\Sigma)$ is a projective 
toric surface with $\rho (X)=2$. 
The map $\pi:N\to \mathbb Z$ induces a Fano contraction morphism 
$\varphi:X\to \mathbb P^1$. 
Let $R$ be the corresponding extremal ray of $\NE(X)$. 
Then $l(R)=1=2-1$. 
Note that $X$ is not a $\mathbb P^1$-bundle over $\mathbb P^1$. 
\end{ex}

\subsection{Examples}\label{f-subsec3.3} 
In this subsection, we will see that the 
estimates in Theorem \ref{f-thm3.2.1} are the best by the following 
examples. 

\begin{ex}\label{f-ex3.3.1}
We use the same notation as in Example \ref{f-ex3.2.7}. 
In Example \ref{f-ex3.2.7}, we put $a=k^2$ and $b=mk+1$ for any positive integers 
$k$ and $m$. 
Then it is obvious that $\gcd(a,b)=1$. 
So, we can apply the construction in Example \ref{f-ex3.2.7}. 
Then we obtain a toric projective birational morphism 
$f:X\to Y$ such that 
$X$ is 
$\mathbb Q$-factorial and $\rho (X/Y)=1$. 
We can easily check that 
\[K_X=f^*K_Y+\left(\frac{1+k^2(n-1)}{mk+1}-1\right)E\] and 
\[-K_X\cdot C=n-1-\frac{m}{k}. \]
Therefore, we see that the minimal lengths of extremal rays do not 
satisfy the ascending chain condition in this birational setting. 
More precisely, 
the minimal lengths of extremal rays can take any values in 
$\mathbb Q\cap (0,n-1)$. For a related topic, see \cite{fujino-ishitsuka}. 
\end{ex}

Let us construct small contraction morphisms with a 
long extremal ray. 

\begin{ex}\label{f-ex3.3.2}
We fix $N=\mathbb Z^n$ with $n\geq 3$. 
Let $\{v_1, \cdots, v_n\}$ be the standard basis of $N$. 
We put  
\[
v_{n+1}=(\underbrace{a, \cdots, a}_{n-k+1}, 
\underbrace{-1, \cdots, -1}_{k-1})
\] 
with $2\leq k\leq n-1$, where $a$ is any positive integer. 
Let 
$\Sigma^+$ be the fan in $\mathbb R^n$ such that  the set of maximal 
cones of $\Sigma^+$ is
\[
\left\{\left\langle \{v_1, \cdots,  v_{n+1}\}\setminus\{v_i\}\right\rangle
\,\left|\,n-k+2\leq i\leq n+1\right\}\right.. 
\]
Let us consider the smooth toric variety $X^+=X(\Sigma^+)$ associated 
to the fan $\Sigma^+$.
We note that the equality 
\[
v_{n-k+2}+\cdots + v_n+v_{n+1}=av_1+\cdots +av_{n-k+1}
\] 
holds. 
We can get an antiflipping contraction 
$\varphi^+: X^+\to W$, 
that is, a $K_{X^+}$-positive small contraction morphism, 
when 
\[
a>\frac{k}{n-k+1}.
\] 
In this case, we have the following flipping diagram 
\[
\xymatrix{
X\ar[dr]_{\varphi}\ar@{-->}[rr]^\phi&& X^+\ar[dl]^{\varphi^+}\\
&W&
}
\]
By construction, $\varphi:X\to W$ is a flipping contraction whose 
exceptional locus is isomorphic to $\mathbb P^{n-k}$. 
The exceptional locus of $\varphi^+$ is isomorphic 
to $\mathbb P^{k-1}$. Of course, $\varphi$ (resp.~$\varphi^+$) 
contracts $\mathbb P^{n-k}$ (resp.~$\mathbb P^{k-1}$) 
to a point in $W$. We can directly check that 
\[
-K_X\cdot C=n-k+1-\frac{k}{a}
\] 
for every torus invariant curve $C$ in the $\varphi$-exceptional locus 
$\mathbb P^{n-k}$. 
\end{ex} 

Example \ref{f-ex3.3.2} shows that the estimate for small contractions in 
Theorem \ref{f-thm3.2.1} is sharp. 

\begin{rem}\label{f-rem3.3.3}
If $(n,k)=(3,2)$ and $a\geq 2$ in Example \ref{f-ex3.3.2}, 
then $\varphi:X\to W$ is a threefold toric flipping contraction 
whose length of the extremal ray is $\geq 3-2=1$. 
We note that the lengths of extremal rays of three-dimensional 
terminal (not necessarily toric) flipping contractions are less than one. 
\end{rem}

%%%%%%%%%%%%%%%%%%%%%%%%

\section{Basepoint-free theorems}\label{f-sec4} 

This section is a supplement to Fujita's freeness conjecture for toric varieties (see  
\cite{fujino-notes}, \cite{fujita}, \cite{laterveer}, \cite{lin}, \cite{mustata}, and 
\cite{payne}) and Fulton's book:~\cite{fulton}. 

\subsection{Variants of Fujita's conjectures for toric varieties}\label{f-subsec4.1} 

One of the most general formulations of 
Fujita's freeness conjecture for 
toric varieties is \cite[Corollary 0.2]{fujino-notes}. 
However, it does not cover the first part of 
\cite[Main theorem A]{lin}. 
So, we give a generalization here with a 
very simple proof. 
It is an easy application of the vanishing theorem (see \cite{fujino-vanishing} and 
\cite{fujino-toric}). 

\begin{thm}[Basepoint-freeness]\label{f-thm4.1.1} 
Let $g:Z\to X$ be a proper toric morphism 
and let $A$ and $B$ be reduced torus invariant Weil divisors 
on $Z$ without common irreducible components. 
Let $f:X\to Y$ be a proper surjective toric morphism 
and let $D$ be an $f$-ample Cartier divisor on $X$.
Then 
\[
R^qg_*(\widetilde 
{\Omega}^a_Z(\log (A+B))(-A))\otimes \mathcal O_X(kD)\] 
is $f$-free, that is, 
\begin{equation*}
f^*f_*(R^qg_*(\widetilde 
{\Omega}^a_Z(\log (A+B))(-A))\otimes \mathcal O_X(kD))  
\to 
R^qg_*(\widetilde 
{\Omega}^a_Z(\log (A+B))(-A))\otimes \mathcal O_X(kD)
\end{equation*}
is surjective 
for every  $a\geq 0$, $q\geq 0$,  and $k\geq \max_{y\in Y}\dim f^{-1}(y) +1$.
\end{thm}

As a very special case, we can recover the following result. 

\begin{cor}[{cf.~\cite[Main theorem A]{lin}}]\label{f-cor4.1.2} 
Let $X$ be an $n$-dimensional 
projective toric variety and let $D$ be an ample Cartier divisor on $X$. Then
the reflexive sheaf $\mathcal O_X(K_X+(n+1)D)$ is generated
by its global sections.
\end{cor}

\begin{proof}
In Theorem \ref{f-thm4.1.1}, 
we assume that 
$g:Z\to X$ is the identity, $A=B=0$, 
$a=\dim X$, $q=0$, and $Y$ is a point. 
Then we obtain the desired statement. 
\end{proof}

\begin{ex}\label{f-ex4.1.3}
Let us consider $X=\mathbb P(1, 1, 1, 2)$. 
Let $P$ be the unique $\frac{1}{2}(1, 1, 1)$-singular point of $X$ 
and let $D$ be an ample Cartier divisor on $X$. 
We can find a torus invariant curve $C$ on $X$ such 
that \[K_X\cdot C\in \frac{1}{2}\mathbb Z\setminus \mathbb Z. \]
Therefore, for every effective Weil divisor $E$ on $X$ such that 
$E\sim K_X+4D$, 
we have $P\in \xSupp E$. 
On the other hand, by Corollary \ref{f-cor4.1.2}, 
the reflexive sheaf $\mathcal O_X(K_X+4D)$ is generated by its 
global sections. 
\end{ex}

Before proving Theorem \ref{f-thm4.1.1}, let us recall the definition of 
the reflexive sheaf $\widetilde {\Omega}^a_{X}(\log (A+B))(-A)$ and 
the vanishing theorem 
in \cite{fujino-toric}. 

\begin{defn}\label{f-def4.1.4}
Let $W$ be any Zariski open set of $Z$ such that $W$ is smooth and 
$\codim _Z(Z\setminus W)\geq 2$. 
In this case,  $A+B$ is a simple normal crossing divisor on 
$W$. On this assumption, 
$\Omega^a_{W}(\log (A+B))$ 
is a well-defined locally free sheaf on $W$. 
Let $\iota:W\hookrightarrow Z$ be the natural 
open immersion. 
Then we put \[\widetilde \Omega^a_Z(\log(A+ B))(-A)=
\iota_*(\Omega^a_W(\log (A+B))\otimes \mathcal O_W(-A))\] for every  
$a\geq 0$. 
It is easy to see that the reflexive sheaf 
\[\widetilde \Omega^a_{Z}(\log (A+B))(-A)\] on $Z$ does not 
depend on the choice of $W$. 
\end{defn}

The next theorem is one of the vanishing 
theorems obtained in \cite{fujino-toric}. 
For the proof and other vanishing theorems, 
see \cite{fujino-vanishing} and \cite{fujino-toric}. 

\begin{thm}[{\cite[Theorem 4.3]{fujino-toric}}]\label{f-thm4.1.5}
Let $g:Z\to X$ be a proper toric morphism and let $A$ and $B$ be reduced 
torus invariant Weil divisors on $Z$ without 
common irreducible components. 
Let $f:X\to Y$ be a proper surjective toric morphism 
and let $L$ be an $f$-ample line bundle on $X$.
Then 
\[
R^pf_*(
R^qg_*(\widetilde 
{\Omega}^a_Z(\log (A+B))(-A))\otimes L)=0\] 
for every 
$p>0$, $q\geq 0$, and $a\geq 0$. 
\end{thm}

Let us prove Theorem \ref{f-thm4.1.1}. 

\begin{proof}[Proof of Theorem \ref{f-thm4.1.1}]
By the vanishing theorem:~Theorem \ref{f-thm4.1.5},
we have 
\[R^pf_*(
R^qg_*(\widetilde 
{\Omega}^a_Z(\log (A+B))(-A))\otimes \mathcal O_X((k-p)D))=0
\]
for every $p>0$, $q\geq 0$, $a\geq 0$, and
$k\geq \max_{y\in Y}\dim f^{-1}(y)+1$.
Since $\mathcal O_X(D)$ is $f$-free, we obtain that 
\[
R^qg_*(\widetilde 
{\Omega}^a_Z(\log (A+B))(-A))\otimes \mathcal O_X(kD)\] 
is $f$-free  
by the Castelnuovo--Mumford 
regularity (see, for example, \cite[Example 1.8.24]{positive}).
\end{proof}

Here, we treat some generalizations of 
Fujita's very ampleness for toric 
varieties as applications of Theorem \ref{f-thm4.1.5}. 
For the details of Fujita's very ampleness for toric 
varieties, see \cite{payne}. 

\begin{thm}\label{f-thm4.1.6} 
Let $f:X\to Y$ be a proper surjective toric 
morphism, let $\Delta$ be a reduced torus invariant 
divisor on $X$ such that $K_X+\Delta$ is Cartier, 
and let $D$ be an $f$-ample Cartier divisor on $X$. 
Then $\mathcal O_X(K_X+\Delta+kD)$ is $f$-very ample for 
every $k\geq \max_{y\in Y} \dim f^{-1}(y) +2$.  
\end{thm}
\begin{proof}
It follows from the Castelnuovo--Mumford 
regularity by the vanishing theorem:~Theorem 
\ref{f-thm4.1.5}. For the details, 
see \cite[Example 1.8.22]{positive}.
\end{proof}

The following corollary is a special case of the above theorem. 

\begin{cor}[{cf.~\cite[Main Theorem B]{lin}}]\label{f-cor4.1.7}  
Let $X$ be an $n$-dimensional projective Gorenstein toric variety 
and let $D$ be an ample Cartier 
divisor on $X$. 
Then $\mathcal O_X(K_X+(n+2)D)$ is very ample. 
\end{cor}

We think that the following theorem has not been 
stated explicitly 
in the literature. 

\begin{thm}[Very ampleness]\label{f-thm4.1.8}  
Let $g:Z\to X$ be a proper toric morphism 
and let $A$ and $B$ be reduced torus invariant Weil divisors 
on $Z$ without common irreducible components. 
Let $f:X\to Y$ be a proper surjective 
toric morphism and 
let $D$ be an $f$-ample 
Cartier divisor on $X$. 
Assume that $
R^qg_*(\widetilde 
{\Omega}^a_Z(\log (A+B))(-A))
$ is locally 
free. 
Then \[R^qg_*(\widetilde 
{\Omega}^a_Z(\log (A+B))(-A))\otimes \mathcal O_X(kD)\] is $f$-very ample 
for every $k\geq \max _{y\in Y} \dim 
f^{-1}(y) +2$.  
\end{thm}
\begin{proof}
The proof of Theorem \ref{f-thm4.1.6} works 
for this theorem since  \[
R^qg_*(\widetilde 
{\Omega}^a_Z(\log (A+B))(-A))
\] 
is locally free by assumption (see \cite[Example 1.8.22]{positive}). 
\end{proof}

\subsection{Lin's problem}\label{f-subsec4.2}
In this subsection, we treat Lin's problem raised in \cite{lin}. 
In \cite[Lemma 4.3]{lin}, she claimed the following lemma, 
which is an exercise in \cite[p.90]{fulton}, without 
proof. 

\begin{lem}\label{f-lem4.2.1} 
Let $X$ be a complete Gorenstein toric variety and let $D$ be an ample 
$($Cartier$)$ divisor. If $\Gamma (X, K+D)\ne 0$ then 
$K+D$ is generated by its global sections. 
In fact, $P_{K+D}$ is the convex hull of $\Int P_D\cap M$. 
\end{lem}

Sam Payne pointed out that Lemma \ref{f-lem4.2.1} does not seem 
to have a known valid proof (see \cite[p.500 Added in proof]{lin}). 
Unfortunately, the following elementary 
example is a counterexample to 
Lemma \ref{f-lem4.2.1}. 
So, Lemma \ref{f-lem4.2.1} is NOT true. 
Therefore, the alternative proof of Theorem A in \cite{lin} 
does not work. 

\begin{ex}\label{f-ex4.2.2}
Let $Y=\mathbb P^n$ and let $P\in Y$ be a torus invariant 
closed point. Let $f:X\to Y$ be the blow-up 
at $P$. 
We put $B=\sum _{i=1} ^{n+1}B_i$, where $B_i$ is a torus 
invariant prime divisor on $Y$ for every $i$. 
Then it is well known that 
$\mathcal O_Y(K_Y)\simeq \mathcal O_Y(-B)$. 
We define $D=f^*B-E$, where $E$ is 
the exceptional divisor of $f$. 
In this case, we have \[K_X=f^*K_Y+(n-1)E\] and 
it is not difficult to see that $D$ is ample. 
Therefore, \[K_X+(n-1)D=f^*(K_Y+(n-1)B)\] is nef, 
that is, $\mathcal O_X(K_X+(n-1)D)$ is generated 
by its global sections. 
We note that $H^0(X, \mathcal O_X(K_X+aD))\ne 0$ for 
every positive integer $a$. However, $K_X+aD$ is not nef for 
any real number $a<n-1$. In particular, 
$H^0(X, \mathcal O_X(K_X+D))\ne 0$ but 
$\mathcal O_X(K_X+D)$ is not generated by its global 
sections. 
\end{ex}

The following theorem is the main theorem of this section. 
It follows from Theorem \ref{f-thm3.2.1}. 

\begin{thm}[see Theorem \ref{f-thm1.1}]\label{f-thm4.2.3}
Let $X$ be a $\mathbb Q$-Gorenstein 
projective toric $n$-fold and let $D$ be an ample Cartier divisor on $X$. 
Then $K_X+(n-1)D$ is pseudo-effective if and only if $K_X+(n-1)D$ is 
nef. 
\end{thm}
\begin{proof}
If $K_X+(n-1)D$ is nef, then $K_X+(n-1)D$ is obviously 
pseudo-effective. 
So, all we have to do is to see that $K_X+(n-1)D$ is 
nef when it is pseudo-effective. 
From now on, we assume that $K_X+(n-1)D$ is pseudo-effective. 
We take a positive rational number $\tau$ such that 
$K_X+\tau D$ is nef but not ample. 
In some 
literature, $1/\tau$ 
is called the {\em{nef threshold of $D$ with respect to $X$}}. It is not difficult to see that 
$\tau$ is rational since the Kleiman--Mori cone is 
a rational polyhedral cone in our case. 
If $\tau \leq n-1$, then the theorem is obvious since 
\[K_X+(n-1)D=K_X+
\tau D+(n-1-\tau )D\] and $D$ is ample. Therefore, 
we assume that $\tau >n-1$. 
We take a sufficiently large positive integer $m$ such 
that $m(K_X+\tau D)$ is Cartier. 
We consider the toric morphism $f:=\Phi _{|m(K_X+\tau D)|}: X\to 
Y$. 
By the definition of $\tau$, $f$ is not an isomorphism. 
Let $R$ be an extremal ray of $\NE(X/Y)$. 
Let $C$ be any integral 
curve on $X$ such that 
$[C]\in R$. Since $(K_X+\tau D)\cdot C=0$, we 
obtain 
$-K_X\cdot C=\tau D\cdot C>n-1$. 
Therefore, $f$ is not birational by Theorem \ref{f-thm3.2.1}. 
Equivalently, $K_X+\tau D$ is not big. 
Thus, the numerical equivalence class of $K_X+\tau D$ is on the boundary 
of the pseudo-effective cone $\PE(X)$ 
of $X$. 
So, \[K_X+(n-1)D=K_X+\tau D-(\tau -(n-1))D\] is outside 
$\PE(X)$. This is a contradiction. 
Therefore, $K_X+(n-1)D$ is nef when 
$K_X+(n-1)D$ is pseudo-effective. 
\end{proof}

As a corollary, we obtain the following result, 
which is a correction of Lemma \ref{f-lem4.2.1}. 
It is a variant of Fujita's freeness conjecture for toric 
varieties. Example \ref{f-ex4.2.2} shows that 
the constant $n-1$ in Corollary \ref{f-cor4.2.4} is the 
best. 

\begin{cor}[see Theorem \ref{f-thm1.1}]\label{f-cor4.2.4}
Let $X$ be a 
Gorenstein projective toric $n$-fold and let $D$ be an 
ample Cartier divisor on $X$. 
If $H^0(X, \mathcal O_X(K_X+(n-1)D))\ne 0$, 
then $\mathcal O_X(K_X+(n-1)D)$ is generated 
by its global sections. 
\end{cor}

\begin{proof}
If $H^0(X, \mathcal O_X(K_X+(n-1)D))\ne 0$, then 
$K_X+(n-1)D$ is obviously pseudo-effective. 
Then, by Theorem \ref{f-thm4.2.3}, $K_X+(n-1)D$ is 
nef. 
If $K_X+(n-1)D$ is a nef Cartier divisor, then 
the complete linear system $|K_X+(n-1)D|$ is basepoint-free by 
Lemma \ref{f-lem2.2.3}. 
\end{proof} 

By Theorem \ref{f-thm3.2.9}, we can check the 
following result. 

\begin{cor}[Corollary \ref{f-cor1.4}]\label{f-cor4.2.5} 
Let $X$ be a $\mathbb Q$-Gorenstein projective 
toric $n$-fold and let $D$ be an ample Cartier divisor 
on $X$. 
If $\rho (X)\geq 3$, then $K_X+(n-1)D$ is always nef. 
More precisely, 
if $\rho (X)\geq 2$ and $X$ is not a $\mathbb P^{n-1}$-bundle 
over $\mathbb P^1$, then $K_X+(n-1)D$ is nef. 
\end{cor} 

\begin{proof}
By Theorem \ref{f-thm3.2.9}, $K_X+(n-1)D$ is nef 
since $\rho (X)\geq 2$ and $X$ is not a 
$\mathbb P^{n-1}$-bundle over $\mathbb P^1$. 
\end{proof}

\subsection{Supplements to Fujita's paper}\label{f-subsec4.3}
This subsection supplements Fujita's paper:~\cite{fujita}. 

We have never seen Corollary \ref{f-cor4.2.4} in the literature.
However, we believe that Fujita could prove Corollary \ref{f-cor4.2.4} without
any difficulties (see Theorem \ref{f-thm4.3.2} below). 
We think that he was not interested in the toric geometry when 
he wrote \cite{fujita}.   
If he was familiar with the toric geometry, then
he would have adopted Example \ref{f-ex4.3.1} in
\cite[(3.5) Remark]{fujita}. 
This example supplements Fujita's remark:~\cite[(3.5) Remark]{fujita}. 
We think that our example is much simpler. 

\begin{ex}\label{f-ex4.3.1}
We fix $N=\mathbb Z^2$. 
We put $e_1=(1,0)$, $e_2=(0,1)$, 
$e_3=(-1,-1)$, and $e_4=(1,2)$. 
We consider the fan $\Sigma$ obtained by subdividing 
$N_{\mathbb R}$ with $e_1$, $e_2$, $e_3$, 
and $e_4$. 
We write $X=X(\Sigma)$, the associated toric variety. 
Then $X$ is Gorenstein and $-K_X$ is ample. 
We put $D=-K_X$. It is obvious that $K_X+D\sim 0$. 
It is easy to see that the Kleiman--Mori cone $\NE(X)$ 
is spanned by the two torus invariant curves 
$E=V(e_4)$ and $E'=V(e_2)$. 
So, we have two extremal contractions. 
By removing $e_4$ from $\Sigma$, we obtain a 
contraction morphism $f:X\to \mathbb P^2$. 
In this case, $E$ is not Cartier although $2E$ is Cartier. 
We note that $-K_X\cdot E=1$. 
The morphism $f$ is the weighted blow-up with the 
weight $(1,2)$ described in Proposition \ref{f-prop3.2.6}. 
Another contraction is obtained by 
removing $e_2$. 
It is a contraction morphism from $X$ to $\mathbb P(1,1,2)$. 
Note that $E'$ is a Cartier divisor on $X$. 
\end{ex}

We close this subsection with the following theorem. 
In Theorem \ref{f-thm4.3.2}, we treat normal Gorenstein projective 
varieties defined over 
$\mathbb C$ with 
only rational singularities, which are not necessarily 
toric. 
So, the readers who are interested only in the toric 
geometry can skip this final theorem. 

\begin{thm}[see \cite{fujita}]\label{f-thm4.3.2} 
Let $X$ be a normal projective variety defined over $\mathbb C$ 
with only rational Gorenstein singularities. 
Let $D$ be an ample Cartier divisor on $X$. 
If $K_X+(n-1)D$ is pseudo-effective with 
$n=\dim X$, 
then $K_X+(n-1)D$ is nef.
\end{thm}
\begin{proof}
We take a positive rational number $\tau$ such that 
$K_X+\tau D$ is nef but not ample. 
It is well known that $\tau\leq n+1$ (see \cite[Theorem 1]{fujita}). 
If $\tau\leq n-1$, then the theorem is obvious. 
Therefore, we assume that $n-1<\tau \leq n+1$. 
If $\tau =n+1$, then $X\simeq \mathbb P^n$ and $\mathcal O_X
(D)\simeq \mathcal O_{\mathbb P^n}(1)$. 
In this case, $K_X+(n-1)D$ is not pseudo-effective. 
Thus, we have $n-1<\tau \leq n$ by 
\cite[Theorem 1]{fujita}. 
By \cite[Theorem 2]{fujita} and its proof, it can be checked easily that $\tau=n$ 
and $K_X+\tau D=K_X+nD$ is nef but is not big. 
Therefore, $K_X+nD$ is on the boundary of the pseudo-effective 
cone of $X$. 
So, $K_X+(n-1)D=K_X+nD-D$ is not pseudo-effective. 
This is a contradiction. 
Anyway, we obtain that 
$K_X+(n-1)D$ is nef if $K_X+(n-1)D$ is pseudo-effective. 
\end{proof}

%%%%%%%%%%

\end{document}